\documentclass[12pt]{amsart}
\usepackage{fancybox}
\usepackage{ascmac}
\usepackage{delarray}
\usepackage{enumerate}
\usepackage{amsmath,amssymb}
\usepackage{color}
\usepackage{comment}

\usepackage{mathrsfs}
\usepackage{paralist}

\makeatletter
 \@addtoreset{equation}{subsection}
\makeatother

\def\re{\mathbb{R}}

\def\N{\mathbb{N}}

\def\eps{\varepsilon}
\def\pd{\partial}
\def\ol{\overline}

\def\la{\lambda}
\def\al{\alpha}

\def\({\left(}
\def\){\right)}
\def\pd{\partial}
\def\void{\phi}
\def\BOX{{\setlength{\unitlength}{1pt}\begin{picture}(8,8)
	\put(1,1){\framebox(6,6)}\end{picture}}\ }
\def\qed{\hfill\BOX\vskip1em\par}

\def\weakto{\rightharpoonup}
\def\intO{\int_{\Omega}}
\def\intdO{\int_{\partial\Omega}}

\def\diver{{\rm div }}

\def\ubp{\underline{u}_{p}}

\def\void{\phi}

\numberwithin{equation}{section}
\newtheorem{theorem}{Theorem}[section]
\newtheorem{corollary}[theorem]{Corollary}
\newtheorem{lemma}[theorem]{Lemma}
\newtheorem{proposition}[theorem]{Proposition}

\numberwithin{theorem}{section}

\begin{document}

\title[Finsler Lane-Emden problem]{Asymptotic behavior of least energy solutions to the Finsler Lane-Emden problem with large exponents}
\author[S. Habibi]{Sadaf Habibi$^1$}
\author[F. Takahashi]{Futoshi Takahashi$^2$}

\date{\today}

\setcounter{footnote}{1}
\footnotetext{
Department of Mathematics, Osaka City University,
3-3-138 Sugimoto, Sumiyoshi-ku, Osaka 558-8585, Japan. \\
e-mail:{\tt habibzaisadaf@gmail.com}}

\setcounter{footnote}{2}
\footnotetext{
Department of Mathematics, Osaka City University,
3-3-138 Sugimoto, Sumiyoshi-ku, Osaka 558-8585, Japan. \\
e-mail:{\tt futoshi@sci.osaka-cu.ac.jp}}

\begin{abstract} 
In this paper we are concerned with the least energy solutions to the Lane-Emden problem driven by an anisotropic operator, 
so-called the Finsler $N$-Laplacian, on a bounded domain in $\re^N$.
We prove several asymptotic formulae as the nonlinear exponent gets large. 

\rm 
\\[0.1cm]
\rm 
\end{abstract}
\maketitle

\bigskip

\section{Introduction}
\label{section:Introduction}

Let $N \ge 2$ be an integer.
In this paper, we study the following Lane-Emden problem driven by an anisotropic operator $Q_N$:
\begin{equation}
\label{E_N}
	\begin{cases}
		&-Q_N u = u^p \quad \mbox{in} \; \Omega, \\
		&u > 0 \quad \mbox{in} \; \Omega, \\
		&u = 0 \quad \mbox{on} \; \pd\Omega,
	\end{cases}
\end{equation}
where $\Omega$ is a smooth bounded domain in $\re^N$,
$p >1$ is any positive number, and $Q_N$ is a quasilinear operator, 
so-called {\it the Finsler $N$-Laplacian}, defined by
\[
	Q_N u = \sum_{i=1}^N \frac{\pd}{\pd x_i} \( H(\nabla u)^{N-1} H_{\xi_i} (\nabla u) \).
\]
Here $H \in C^2(\re^N \setminus \{ 0 \})$ is any norm on $\re^N$ and $H_{\xi_i}(\xi) = \frac{\pd H(\xi)}{\pd \xi_i}$.
We assume that $H^N \in C^1(\re^N)$ and ${\rm Hess} \( H^N(\xi) \)$ is positive definite for any $\xi \in \re^N$, $\xi \ne 0$.
%
%
Note that $Q_N u$ can be written as 
\[
	Q_N u = \diver \( \nabla_{\xi} ( \frac{1}{N} H(\xi)^N ) \Big|_{\xi = \nabla u} \) 
	= \sum_{i,j =1}^N a_{ij}(\nabla u) \frac{\pd^2 u}{\pd x_i \pd x_j},
\]
where $a_{ij}(\nabla u) = {\rm Hess} (\frac{1}{N} H^N(\xi))_{i,j} \Big|_{\xi = \nabla u}$.
If $H(\xi) = |\xi|$ (the Euclidean norm), 
then $Q_N u$ coincides with the $N$-Laplacian $\Delta_N u = \diver (|\nabla u|^{N-2} \nabla u)$ of a function $u$.
In this case, the problem \eqref{E_N} was treated by Ren and Wei \cite{Ren-Wei(TAMS)} \cite{Ren-Wei(PAMS)} when $N=2$, 
and in \cite{Ren-Wei(JDE)} for general $N \ge 2$. 
Ren and Wei \cite{Ren-Wei(JDE)} considered the least energy solution $u_p$ of the following quasilinear problem
\begin{align*}
	\begin{cases}
		&-\Delta_N u = u^p \quad \mbox{in} \; \Omega, \\
		&u > 0 \quad \mbox{in} \; \Omega, \\
		&u = 0 \quad \mbox{on} \; \pd\Omega,
	\end{cases}
\end{align*}
where $\Omega$ is a smooth bounded domain in $\re^N$.
They studied the asymptotic behavior of $u_p$ as the nonlinear exponent $p \to \infty$,
and proved that the least energy solutions remain bounded in $L^{\infty}$-norm regardless of $p$.
When the dimension $N = 2$, they showed that the least energy solutions must develop one ``peak" in the interior of $\Omega \subset \re^2$,
that is, the shape of graph of $u_p$ looks like a single spike as $p \to \infty$.
Moreover they showed that this peak point must be a critical point of the Robin function of the domain.
For other generalizations of this problem to various situations, see for example, \cite{TF(OJM)}, \cite{TF(CVPDE)}, \cite{TF(JMAA)}, \cite{Santra}, \cite{Santra-Wei}.

%
%

Now, main aim of the paper is to extend the results of Ren and Wei \cite{Ren-Wei(TAMS)}, \cite{Ren-Wei(PAMS)}, \cite{Ren-Wei(JDE)} to the anisotropic problem \eqref{E_N}.

As in \cite{Ren-Wei(TAMS)}, \cite{Ren-Wei(PAMS)}, \cite{Ren-Wei(JDE)},
we restrict our attention to the least energy solutions to \eqref{E_N} constructed as follows:

Consider the constrained minimization problem:
\begin{equation}
\label{Cp}
	C_p = \inf \{ \intO H(\nabla u)^N dx: u \in W^{1,N}_0(\Omega), \int_{\Omega} |u|^{L^{p+1}(\Omega)} dx =1 \}.
\end{equation}
Since the Sobolev imbedding $W^{1,N}_0(\Omega) \hookrightarrow L^{p+1}(\Omega)$ is compact for any $p>1$,
we have at least one minimizer $\ubp$ for the problem (\ref{Cp}), where 
$\ubp \in W^{1,N}_0(\Omega), \Vert \ubp \Vert_{p+1} = 1$.
As $|\ubp| \in W^{1,N}_0(\Omega)$ also achieves $C_p$, we may assume $\ubp >0$.
Note that $Q_N(cu) = c^{N-1} Q_N(u)$ for a constant $c > 0$. 
Thus if we define 
\[
	u_p = C_p^{\frac{1}{p+1-N}} \ubp,
\]
then $u_p$ solves $\eqref{E_N}$ and $C_p = \intO H(\nabla u_p)^N dx / (\intO |u_p|^{p+1}dx)^{\frac{N}{p+1}}$. 
Standard regularity argument implies that any weak solution $u \in W^{1,N}_0(\Omega)$ satisfies
$u \in C^{1, \al}(\ol{\Omega})$ for some $\al \in (0,1)$.
We call $u_p$ the {\it least energy solution} to \eqref{E_N}.

Our first result is the following $L^{\infty}$-bound of least energy solutions.

\begin{theorem}
\label{Theorem:bound}
Let $u_p$ be a least energy solution to \eqref{E_N}. Then there exist $C_1, C_2$ (independent of $p$),
such that 
\[
	0 < C_1 \le \| u_p \|_{L^{\infty}(\Omega)} \le C_2 < \infty
\]
for $p$ large enough.

Furthermore, we have
\begin{align*}
	\lim_{p \to \infty} p^{N-1} \intO H(\nabla u_p)^N dx = \lim_{p \to \infty} p^{N-1} \intO u_p^{p+1} dx = \(\frac{Ne \beta_N}{N-1}\)^{N-1}
\end{align*}
where $\beta_N = N(N\kappa_N)^{\frac{1}{N-1}}$, $\kappa_N = |\mathcal{W}|$ is the volume (with respect to the $N$-dimensional Hausdorff measure) of the unit Wulff ball associated with the dual norm $H^0$ of $H$: 
\[
	\mathcal{W} = \{ x \in \re^N : H^0(x) < 1 \}.
\]
\end{theorem}

On the asymptotic behavior of the $L^{\infty}$-norm of $u_p$, we have
\begin{theorem}
\label{Theorem:limsup}
Let $u_p$ be a least energy solution to \eqref{E_N}.
Then it holds that
\[
	1 \le \limsup_{p \to \infty} \| u_p \|_{L^{\infty}(\Omega)} \le e^{\frac{N-1}{N}}.
\]
\end{theorem}

To state further results, we need some definitions.
Set 
\begin{equation}
\label{v_p}
	v_p = \frac{u_p}{(\intO u_p^p dx)^{\frac{1}{N-1}}}.
\end{equation}
Then $v_p$ is a weak solution of
\begin{equation}
\label{v_p_eq}
	\begin{cases}
	-Q_N v_p = f_p(x) = \frac{u_p^p}{\intO u_p^p dx} \quad \text{in} \; \Omega, \\ 
	v_p > 0 \quad \text{in} \; \Omega, \\ 
	v_p = 0 \quad \text{on} \; \pd\Omega.
	\end{cases}
\end{equation}
For $X, Y \in \re^N, X, Y \ne 0, X \ne Y$, we put
\[
	d(X, Y) = \frac{\( H^{N-1}(X) (\nabla_{\xi} H)(X) - H^{N-1}(Y) (\nabla_{\xi} H)(Y)\) \cdot (X - Y)}{H^N(X-Y)},
\]
and
\begin{equation}
\label{d_N}
	d_N = \inf \{ d(X, Y) \ | \ X, Y \in \re^N, X, Y \ne 0,  X \ne Y \},
\end{equation}
where $X \cdot Y = \sum_{j=1}^N X_j Y_j$ denotes the usual inner product for $X, Y \in \re^N$.
As in \cite{Wang-Xia} Lemma 5.1, we can obtain the estimate
\[
	\min \{ \frac{\la}{\beta^N}, 1 \} \le d_N \le 1
\]
where $\la$ is the least eigenvalue of ${\rm Hess} \(\frac{1}{N} H^N(\xi) \)$, which is positive by the assumption \eqref{Hess}
and $\beta$ is as in \eqref{alpha_beta}, see \S \ref{section:Notation}. 
Also define
\begin{align}
\label{L_0L_1}
	L_0 = \limsup_{p \to \infty} \frac{p \(\intO u_p^p dx \)^{\frac{1}{N-1}}}{\( \frac{N}{N-1} e^{\frac{N-1}{N}} \)}, \quad L_1 = d_N^{-\(\frac{1}{N-1}\)} L_0.
\end{align}
For a sequence $v_{p_n}$ of $v_p$, we define the {\it blow-up set} $S$ of $\{ v_{p_n} \}$ as usual:
\begin{equation*}
\label{S}
	S = \{ x \in \overline{\Omega}: \exists \mbox{a subsequence} \; v_{p_n^{'}}, \exists \{ x_n \} \subset \Omega \; \mbox{s.t.} \ x_n \to x \; \mbox{and} \; v_{p_n^{'}}(x_n) \to \infty \}.
\end{equation*}
In the following, $\sharp A$ denotes the cardinality of a set $A$ and $[ \cdot ]$ denotes the Gauss symbol.

\begin{theorem}
\label{Theorem:blowup}
Let $\Omega \subset \re^N$ be a smooth bounded domain.
Then for any sequence $v_{p_n}$ of $v_p$ with $p_n \to \infty$, the blow-up set $S$ of $v_{p_n}$ is non-empty.
Also there exists a subsequence (still denoted by $v_{p_n}$) such that the estimate
\[
	\sharp (S \cap \Omega) \le \left[ \frac{e^{\frac{N-1}{N}}}{d_N} \right]
\]
holds true for this subsequence.

Assume $S \cap \Omega = \{ x_1, \cdots, x_k \} \subset \Omega$.
Then we have
\begin{enumerate}
\item[(i)]
\[
	f_n = \frac{u_{p_n}^{p_n}}{\intO u_{p_n}^{p_n} dx} \stackrel{*}{\weakto} \sum_{i=1}^k \gamma_i \delta_{x_i} 
\]
in the sense of Radon measures of $\Omega$, 
where 
\[
	\gamma_i \ge \( \frac{\beta_N}{L_1} \)^{N-1}
\]
and $\sum_{i=1}^k \gamma_i \le 1$.
\item[(ii)]
$v_{p_n} \to G$ in $C^1_{loc}(\Omega \setminus (S \cap \Omega))$ for some function $G$ satisfying
\[
	\begin{cases}
	-Q_N G = 0 \quad &\text{in} \, \Omega \setminus (S \cap \Omega), \\ 
	G = + \infty \quad &\text{on} \, S \cap \Omega, \\ 
	G = 0 \quad &\text{on} \,\pd\Omega \setminus (\pd\Omega \cap S).
	\end{cases}
\]
\item[(iii)] $\| u_{p_n} \|_{L^{\infty}(K)} \to 0$ as $n \to \infty$ for any compact set $K \subset \Omega \setminus (S \cap \Omega)$.
\end{enumerate}
\end{theorem}

In \cite{Ren-Wei(JDE)}, Ren and Wei obtained an estimate of the number of interior blow-up set
\[
	\sharp (S \cap \Omega) \le \left[ \frac{1}{d_N} \( \frac{N}{N-1} \)^{N-1} \right]
\]
when $H(\xi) = |\xi|$ case.
Since $e^x < \frac{1}{1-x}$ for $x \in (0,1)$, we check that $e^{\frac{N-1}{N}} < \( \frac{N}{N-1} \)^{N-1}$ for all $N \ge 2$.
Thus the estimate in Theorem \ref{Theorem:blowup} is better than that in \cite{Ren-Wei(JDE)} 
even when $H(\xi)$ coincides with the Euclidean norm $|\xi|$.
Also, Theorem \ref{Theorem:limsup} seems new even for $H(\xi) = |\xi|$ and $N > 2$ case.

Finally, we prove that if the blow-up set consists of one point, it must be an interior point of $\Omega$.
\begin{theorem}
\label{Theorem:one_point_blowup}
Assume $\sharp S = 1$ and $S = \{ x_0 \}$, $x_0 \in \ol{\Omega}$.
Then $x_0 \in \Omega$ must hold.
\end{theorem}

The organization of the paper is as follows: 
In \S \ref{section:Notation}, we recall basic properties of the Finsler norm and collect useful lemmas about the Finsler $N$-Laplacian.
In \S \ref{section:asymptotics}, we obtain asymptotic formula for $C_p$ as $p \to \infty$, and prove the latter half part of Theorem \ref{Theorem:bound}. 
In \S \ref{section:bound}, we prove the $L^{\infty}$-bound of least energy solutions in Theorem \ref{Theorem:bound}.
In \S \ref{section:limsup_bound}, we prove Theorem \ref{Theorem:limsup} using an argument by Adimurthi and Grossi \cite{Adimurthi-Grossi}.
In \S \ref{section:blowup}, we prove Theorem \ref{Theorem:blowup}. We use a notion of $(L, \delta)$-regular, or irregular points, which was originally introduced by Brezis and Merle \cite{Brezis-Merle}.
Finally in \S \ref{section:one_point_blowup}, we prove Theorem \ref{Theorem:one_point_blowup} by using a local Pohozaev identity and an idea by Santra and Wei \cite{Santra-Wei}.

\section{Notations and basic properties}
\label{section:Notation}

Let $H$ be any {\it norm} on $\re^N$, i.e., 
$H$ is convex, $H(\xi) \ge 0$ and $H(\xi) = 0$ if and only if $\xi = 0$, and $H$ satisfies 
\begin{equation}
\label{homo_H}
	H(t \xi) = |t| H(\xi), \quad \forall \xi \in \re^N, \, \forall t \in \re.
\end{equation}
By \eqref{homo_H}, $H$ must be even: $H(-\xi) = H(\xi)$ for all $\xi \in \re^N$.
Throughout of the paper, we also assume that $H \in C^2(\re^N \setminus \{ 0 \})$, $H^N \in C^1(\re^N)$, and
\begin{equation}
\label{Hess}
	\text{${\rm Hess} \( H^N(\xi) \)$ is positive definite for any $\xi \in \re^N$, $\xi \ne 0$.}
\end{equation}
Since all norms on $\re^N$ are equivalent to each other, we see the existence of positive constants $\alpha$ and $\beta$ such that
\begin{equation}
\label{alpha_beta}
	\alpha |\xi| \le H(\xi) \le \beta |\xi|, \quad \xi \in \re^N.
\end{equation}
The dual norm of $H$ is the function $H^0: \re^N \to \re$ defined by 
\[
	H^0(x) = \sup_{\xi \in \re^N \setminus \{ 0 \}} \frac{\xi \cdot x}{H(\xi)}. 
\]
It is well-known that $H^0$ is also a norm on $\re^N$ and satisfies the inequality
\[
	\frac{1}{\beta} |x| \le H^0(x) \le\frac{1}{\alpha} |x|, \quad \forall x \in \re^N.
\]
The set
\[
	\mathcal{W} = \{ x \in \re^N \,: \, H^0(x) < 1 \}
\]
is called the {\it Wulff ball}, or the $H^0$-unit ball, and we denote $\kappa_N = \mathcal{H}^N(\mathcal{W})$,
where $\mathcal{H}^N$ denotes the $N$-dimensional Hausdorff measure on $\re^N$.
We also denote $\mathcal{W}_r = \{ x \in \re^N \,|\, H^0(x) < r \}$ for any $r > 0$.

For a domain $\Omega \subset \re^N$ and a Borel set $E \subset \re^N$, the {\it anisotropic $H$-perimeter} of a set $E$ with respect to $\Omega$
is defined as
\[
	P_H(E, \Omega) = \sup \left\{ \int_{E \cap \Omega} \diver \sigma dx \ : \ \sigma \in C_0^{\infty}(\Omega, \re^N), H^0(\sigma(x)) \le 1 \right\}.
\]  
If $E$ is Lipschitz, then it holds $P_H(E, \Omega) = \int_{\Omega \cap \pd^* E} H(\nu) d\mathcal{H}^{N-1}$, 
where $\pd^* E$ denotes the reduced boundary of the set $E$ and $\nu(x)$ is the measure theoretic outer unit normal of $\pd^* E$ (see \cite{Evans-Gariepy}).
Also we have $P_H(\mathcal{W}, \re^N) = N \kappa_N$.
For more explanation about the anisotropic perimeter, see \cite{AFTL} and {\cite{Belloni-Ferone-Kawohl}.

Here we just recall some properties of $H$ and $H^0$.
These will be proven by using the homogeneity property of $H$ and $H^0$,
see \cite{Bellettini-Paolini} Lemma 2.1, and Lemma 2.2.

\begin{proposition}
\label{Prop:identities} 
Let $H$ be a Finsler norm on $\re^N$. 
Then the following properties hold true:
\begin{enumerate}
\item[\rm(1)] $|\nabla_{\xi} H(\xi)| \le C$ for any $\xi \ne 0$.
\item[\rm(2)] $\nabla_{\xi} H(\xi) \cdot \xi = H(\xi)$, $\nabla_x H(x) \cdot x = H(x)$ for any $\xi \ne 0$, $x \ne 0$. 
\item[\rm(3)] $\(\nabla_{\xi} H \)(t \xi) = \frac{t}{|t|} \( \nabla_{\xi} H \)(\xi)$ for any $\xi \ne 0$, $t \ne 0$.
\item[\rm(4)] $H\( \nabla H^0(x) \) = 1$. $H^0\( \nabla_{\xi} H(\xi) \) = 1$.
\item[\rm(5)] $H^0(x) \( \nabla_{\xi} H \) \( \nabla_x H^0(x) \) = x$.
\end{enumerate}
\end{proposition}

Finally, given a smooth function $u$ on $\re^N$, {\it the Finsler Laplace operator} of $u$ (associated with $H$) is defined by
\begin{align*}
	Q u(x) &= {\rm div} \( H(\nabla u(x)) \( \nabla_{\xi} H \)(\nabla u(x)) \) \\
	&= \sum_{j=1}^N \frac{\pd}{\pd x_j} \( H(\xi) H_{\xi_j}(\xi) \Big|_{\xi = \nabla u(x)} \)
\end{align*}
and, more generally, for any $1< q <\infty$, {\it the Finsler $q$-Laplace operator} $Q_q$ by
\[
	Q_q u(x) = \mbox{div} \( H^{q-1}(\nabla u(x))(\nabla_{\xi} H)(\nabla u(x)) \).
\]
If we assume that ${\rm Hess} (H^q(\xi))$ is positive definite on $\re^N \setminus \{ 0 \}$, $Q_q$ becomes a uniformly elliptic operator locally.
The Finsler $q$-Laplacian has been widely investigated in literature by many authors in different settings,
see \cite{AFMTV}, \cite{Belloni-Ferone-Kawohl}, \cite{Cianci-Salani}, \cite{CFR}, \cite{CFV}, \cite{DP-B}, \cite{DP-B-G}, \cite{DP-G}, \cite{Ferone-Kawohl}, \cite{Mercaldo-Sano-TF} \cite{Zhou-Zhou} and the references therein.

We collect here several useful facts.

\begin{theorem}{\rm (Finsler Trudinger-Moser inequality \cite{Wang-Xia})}
\label{Theorem:TM} 
Let $\Omega$ be a bounded domain in $\re^N$, $N \ge 2$. Let $u \in W^{1,N}_0(\Omega)$ satisfy $\intO H(\nabla u)^N dx \le 1$.
Then there exists a constant $C$ depending only on the dimension $N$ such that
\[
	\intO \exp \( \beta |u|^{\frac{N}{N-1}} \) dx \le C |\Omega|
\]
holds for any $\beta \le \beta_N = N(N\kappa_N)^{\frac{1}{N-1}}$.
Furthermore, $\beta_N$ is optimal in the sense that there exists a sequence $\{ u_n \} \subset W^{1,N}_0(\Omega)$ with $\intO H(\nabla u_n)^N dx \le 1$, 
such that $\intO \exp \( \beta |u_n|^{\frac{N}{N-1}} \) dx \to +\infty$ as $n \to \infty$ for $\beta > \beta_N$.
\end{theorem}

Next is the unique existence of the Green function for the Finsler $p$-Laplacian.
\begin{theorem}{\rm (\cite{Wang-Xia})}
\label{Theorem:Green}
Let $\Omega \subset \re^N$ be a bounded domain containing the origin. 
Define $\Omega^* = \Omega \setminus \{ 0 \}$ and
\[
	\Gamma(x) = 
	\begin{cases}
	&C(p,N) (H^0(x))^{\frac{p-N}{p-1}} \quad \text{for} \quad 1 < p < N, \\
	&C(N) \log \frac{1}{H^0(x)} \quad \text{for} \quad p = N,
	\end{cases}
\]
where $C(p,N) = \frac{p-N}{p-1} (N\kappa_N)^{-\frac{1}{p-1}}$ and $C(N) = (N\kappa_N)^{-\frac{1}{N-1}}$.
Then there exists a unique function $G(\cdot, 0) \in C^{1, \al}(\Omega^*)$ with $|\nabla G| \in L^{p-1}(\Omega)$, $G/\Gamma \in L^{\infty}(\Omega)$, satisfying
\begin{align*}
	\begin{cases}
		-Q_p G(\cdot,0) = \delta_0 &\quad \text{in} \ \Omega, \\
		G(\cdot,0) = 0 &\quad \text{on} \ \pd\Omega.
	\end{cases}
\end{align*}
Moreover, $g = G - \Gamma$ satisfies $g \in C(\Omega)$ and $\lim_{x \to 0} H^0(x) \nabla g(x) = 0$.
\end{theorem}

We recall here useful regularity estimates which are valid for the Finsler $N$-Laplacian equations, under the assumption \eqref{Hess};
see Serrin \cite{Serrin}, Tolksdorf \cite{Tolksdorf}, DiBenedetto \cite{DiBenedetto} and Lieberman \cite{Lieberman}.

\begin{theorem}
\label{Theorem:Serrin}
Let $\Omega \subset \re^N$ be a smooth bounded domain.
Then the following statements are true.
\begin{enumerate}
\item[(1)] 
Let $u \in W^{1,N}(\Omega)$ be a weak solution of
$-Q_N u = f$ in $\Omega$, where $f \in L^q(\Omega)$ for some $q > 1$.
Then for any subdomain $\Omega' \subset\subset \Omega$, there exists a constant $C = C(\Omega, \Omega', q, N) > 0$ such that
\[
	\| u \|_{L^{\infty}(\Omega')} \le C \( \| f \|_{L^q(\Omega)} + \| u \|_{L^N(\Omega)} \)
\]
holds.
\item[(2)] 
Let $u \in W^{1,N}(\Omega)$ be a weak solution of $-Q_N u = f$ in $\Omega$.
Suppose $\| u \|_{L^{\infty}(\Omega)} \le a$ and $\| f \|_{L^{\infty}(\Omega)} \le b$ for some $a, b <\infty$. 
Then $u \in C^{1,\al}_{loc}(\Omega)$ for some $\al \in (0,1)$ 
and for any subdomain $\Omega' \subset\subset \Omega$, there exists a constant $C = C(\Omega, \Omega', a, b, \al) > 0$ such that
\[
	\| u \|_{C^{1,\al}(\Omega')} \le C
\]
holds.
If, in addition, $u$ satisfies the Dirichlet boundary condition $u = \phi$ on $\pd\Omega$ where $\phi \in C^{1,\beta}(\pd\Omega)$, $\beta \in (0,1)$,
then $u \in C^{1,\al}(\ol{\Omega})$ for some $\al \in (0,1)$ holds. 
\item[(3)] {\rm (Harnack inequality)} 
Let $u \in W^{1,N}(\Omega)$ be a nonnegative weak solution of $-Q_N u = f$ in $\Omega$. 
Suppose $\| f \|_{L^q(\Omega)} \le b$ for some $q > 1$. 
Then for any subdomain $\Omega' \subset\subset \Omega$, there exists a constant $C = C(\Omega, \Omega', q, b) > 0$ such that
\[
	\sup_{x \in \Omega'} u(x) \le C \( 1 + \inf_{x \in \Omega'} u(x) \)
\]
holds.
\end{enumerate}
\end{theorem}

Next is the result from \cite{Xie-Gong} (Theorem 1.1 and Theorem 1.2).

\begin{theorem}{\rm (Finsler Brezis-Merle type inequality \cite{Xie-Gong})}
\label{Theorem:BM}
Let $\Omega$ be a bounded domain in $\re^N$, $N \ge 2$.
\begin{enumerate}
\item[(1)] 
Suppose $u$ is a weak solution to
\begin{align*}
	\begin{cases}
		-Q_N u = f(x) &\quad \text{in} \ \Omega, \\
		u = 0 &\quad \text{on} \ \pd\Omega,
	\end{cases}
\end{align*}
where $f \in L^1(\Omega)$. 
Then for any $\eps \in (0, \beta_N)$ where $\beta_N = N(N\kappa_N)^{\frac{1}{N-1}}$,
it holds that
\[
	\intO \exp \( \frac{(\beta_N - \eps) |u(x)|}{\| f \|_{L^1(\Omega)}^{\frac{1}{N-1}}} \) dx \le \frac{\beta_N}{\eps} |\Omega|.
\]
\item[(2)] 
Suppose $u$ and $v$ are weak solutions to
\[
	-Q_N u = f(x) > 0 \quad \text{in} \ \Omega
\]
and
\[
	-Q_N v = 0 \quad \text{in} \ \Omega \quad v = u \quad \text{on} \ \pd\Omega,
\]
respectively. Then for any $\eps \in (0, \beta_N)$, we have
\[
	\intO \exp \( \frac{(\beta_N - \eps) d_N^{\frac{1}{N-1}} |u(x) - v(x)|}{\| f \|_{L^1(\Omega)}^{\frac{1}{N-1}}} \) dx \le \frac{|\Omega|}{\eps},
\]
where $d_N$ is defined in \eqref{d_N}.
\end{enumerate}
\end{theorem}

Next is the Pohozaev identity for the Finsler $q$-Laplacian problem without the boundary condition.
This is a special case of much more general identity proved in \cite{DG-M-S}.
The identity below is known to hold for solutions in $C^1(\ol{\Omega}) \cap C^2(\Omega)$.
The important point is that we can remove the condition $u \in C^2(\Omega)$ 
with the cost of the convexity and the $C^1(\re^N)$-regularity of the map $\re^N \ni \xi \mapsto H^q(\xi)$. 
This improvement is crucial for the application to the Finsler Laplacian problem,
since the best possible regularity result of solutions is $C^{1,\al}$, not $C^2$, see Theorem \ref{Theorem:Serrin}.

\begin{theorem}{\rm (\cite{DG-M-S})}
\label{Theorem:Pohozaev}
Let $1 < q < \infty$.
Let $u \in C^1(\ol{\Omega})$ be a weak solution of $-Q_q u = f(u)$ in $\Omega$,
where $\Omega \subset \re^N$ is a domain with the boundary of class $C^1$, and $f \in C(\re, \re)$.
Assume the map $\re^N \ni \xi \mapsto H^q(\xi)$ is convex and belongs to $C^1(\re^N)$.
Then the identity
\begin{align*}
	N \intO F(u) dx &- \( \frac{N-q}{q} \) \intO H^q(\nabla u) dx \\
	&= \int_{\pd\Omega} F(u)(x-y) \cdot \nu(x) ds_x \\
	&- \frac{1}{q} \int_{\pd\Omega} H^q(\nabla u)(x-y) \cdot \nu(x) ds_x \\
	&+ \int_{\pd\Omega} \( H^{q-1}(\nabla u) (\nabla_{\xi} H)(\nabla u) \cdot \nu(x) \) ((x-y) \cdot \nu(x)) ds_x \\
\end{align*}
holds true for any $y \in \re^N$. 
Here $\nu$ is the outer unit normal of $\pd\Omega$ and $F(s) = \int_0^s f(t) dt$.
\end{theorem}

\begin{proof}
Indeed, since $\mathcal{L}(x, s, \xi) = \frac{1}{q} H^q(\xi) - F(s)$ is of the ``splitting" form, $F \in C^1(\re)$, and $\xi \mapsto H^q(\xi)$ is convex and in $C^1(\re^N)$,
Lemma 5, thus the equation (3) in \cite{DG-M-S} holds as it is.
Also, if we do not impose the boundary condition $u = 0$ on $\pd\Omega$ (and put $f = 0$ there) in Lemma 2 in \cite{DG-M-S}, 
we obtain the identity
\begin{align*}
	&\intdO \mathcal{L}(x, u, \nabla u) (h \cdot \nu) ds_x -  \sum_{i,j=1}^N \intdO h_j D_{\xi_i} \mathcal{L}(x, u, \nabla u) D_{x_j} u \nu_i ds_x \\
	&\hspace{2em} = \intO (\diver h) \mathcal{L}(x, u, \nabla u) dx -  \sum_{i,j=1}^N \intdO D_i h_j D_{\xi_i} \mathcal{L}(x, u, \nabla u) D_{x_j} u dx
\end{align*}
for every $h \in C^1(\ol{\Omega}, \re^N)$.
Inserting $h(x) = x$ leads to the claim.
\end{proof}

Finally, we prove the following simple lemma.
\begin{lemma}
\label{Lemma:weak}
Let $u \in W^{1,N}_0(\Omega)$ be a weak solution to $-Q_N u = f(u)$ in $ \Omega \subset \re^N$, where $f:\re \to \re$ is continuous.
Let $a, c > 0$, $d \in \re$ and $b \in \re^N$.
Then $v(x) = c u(ax + b) + d$, $x \in \Omega_{a,b} = \frac{\Omega - b}{a}$ is a weak solution to
\[
	-Q_N v = a^N c^{N-1} f\(\frac{v-d}{c}\) \quad \text{in} \ \Omega_{a, b}, \quad v = 0 \quad \text{on} \ \pd \Omega_{a,b}.
\]
\end{lemma}

\begin{proof}
For $x \in \Omega_{a,b}$, put $y = ax + b \in \Omega$.
Then for any $\phi \in C_0^{\infty}(\Omega_{a,b})$, $\tilde{\phi}(y) = \phi(x)$ belongs to $C_0^{\infty}(\Omega)$.
Therefore we have
\begin{align*}
	&\int_{\Omega_{a,b}} H^{N-1}(\nabla v(x)) (\nabla_{\xi} H)(\nabla v(x)) \cdot \nabla \phi(x) dx \\
	&= \int_{\Omega_{a,b}} H^{N-1}(c a (\nabla u)(ax + b)) (\nabla_{\xi} H)(ca (\nabla u) (ax + b)) \cdot \nabla \phi(x) dx \\
	&= \int_{\Omega} c^{N-1} a^{N-1} H^{N-1}(\nabla u(y)) (\nabla_{\xi} H)((\nabla u(y)) \cdot a \nabla \tilde{\phi}(y) a^{-N} dy \\
	&= c^{N-1} \int_{\Omega} H^{N-1}(\nabla u(y)) (\nabla_{\xi} H)(\nabla u(y)) \cdot \nabla \tilde{\phi}(y) dy \\
	&= c^{N-1} \int_{\Omega} f(u(y)) \tilde{\phi}(y) dy \\
	&= c^{N-1} \int_{\Omega_{a,b}} f\( \frac{v(x) - d}{c} \) \phi(x) a^N dx,
\end{align*}
where we have used \eqref{homo_H} and Proposition \ref{Prop:identities} (3). 
Thus we see
\begin{align*}
	\int_{\Omega_{a,b}} H^{N-1}(\nabla v(x)) (\nabla_{\xi} H)(\nabla v(x)) \cdot \nabla \phi(x) dx = a^N c^{N-1} \int_{\Omega_{a,b}} f\( \frac{v(x) - d}{c} \) \phi(x) dx.
\end{align*}
This holds true for any $\phi \in C_0^{\infty}(\Omega_{a,b})$, which implies Lemma.
\end{proof}


\section{Asymptotic estimate for $C_p$}
\label{section:asymptotics}

In this section, first by using the Finsler Trudinger-Moser inequality Theorem \ref{Theorem:TM}, we establish the refined Sobolev embedding. 

\begin{lemma}
\label{Lemma:refined_Sobolev}
Let $\Omega \subset \re^N$ be a bounded domain.
For any $t \ge 2$, there exists $D_t >0$ such that for any $u \in W^{1,N}_0(\Omega)$,
\[
	\Vert u \Vert_{L^t(\Omega)} \le D_t t^{\frac{N-1}{N}} \Vert H(\nabla u) \Vert_{L^N(\Omega)}
\]
holds true.
Furthermore, we have
\[
	\lim_{t \to \infty} D_t = \( \frac{1}{N \kappa_N^{1/N}} \) \( \frac{N-1}{Ne} \)^{\frac{N-1}{N}}.
\]
\end{lemma}

\begin{proof}
Let $u \in W^{1,N}_0(\Omega)$.
By the elementary inequality $\frac{x^s}{\Gamma(s+1)} \le e^x$ for $x \ge 0$ and $s \ge 0$,
where $\Gamma(s)$ is the Gamma function,
and the Finsler Trudinger-Moser inequality,
we have
\begin{align*}
	&\frac{1}{\Gamma(\frac{N-1}{N} t+1)} \intO |u|^t dx \\
	&= \frac{1}{\Gamma(\frac{N-1}{N} t+1)} \intO \( \beta_N \( \frac{|u|}{\Vert H(\nabla u) \Vert_{L^N(\Omega)}} \)^{\frac{N}{N-1}} \)^{\frac{N-1}{N} t} dx
	\beta_N^{-\frac{N-1}{N}t} \Vert H(\nabla u) \Vert_{L^N(\Omega)}^t \\
	&\le \intO \exp \( \( \beta_N \frac{|u(x)|}{\Vert \nabla u \Vert_{L^N(\Omega)}} \)^{\frac{N}{N-1}} \) dx \beta_N^{-\frac{N-1}{N}t} \Vert H(\nabla u) \Vert_{L^N(\Omega)}^t  \\
	&\le C |\Omega| \beta_N^{-\frac{N-1}{N}t} \Vert H(\nabla u) \Vert_{L^N(\Omega)}^t.
\end{align*}
Put
\[
	D_t = \Gamma\(\frac{N-1}{N} t+1\)^{1/t} C^{1/t} |\Omega|^{1/t} \beta_N^{-\frac{N-1}{N}} t^{-\frac{N-1}{N}}. 
\]
Then we have 
\[ 
	\Vert u \Vert_{L^t(\Omega)} \le D_t t^{\frac{N-1}{N}} \Vert H(\nabla u) \Vert_{L^N(\Omega)}.
\]
Stirling's formula implies that 
\[
	\(\Gamma\(\frac{(N-1)t}{N}+1\)\)^{\frac{1}{t}} \sim \(\frac{N-1}{Ne}\)^{\frac{N-1}{N}} t^{\frac{N-1}{N}} 
\]
as $t \to \infty$.
So we have
\[	
	\lim_{t \to \infty} D_t = \beta_N^{-\frac{N-1}{N}} \(\frac{N-1}{Ne}\)^{\frac{N-1}{N}} = \( \frac{1}{N \kappa_N^{1/N}} \) \( \frac{N-1}{Ne} \)^{\frac{N-1}{N}},
\]
which is a desired result.
\end{proof}

Recall that $C_p$ is defined in (\ref{Cp}).
Using the above Lemma and energy comparison, we get the following.

\begin{proposition}
\label{Prop:C_p_asymptotics}
We have
\[
	\lim_{p \to \infty} p^{N-1} C_p =  \( \frac{Ne}{N-1} \beta_N \)^{N-1}.
\]
where $\beta_N = N(N\kappa_N)^{\frac{1}{N-1}}$.
\end{proposition}

\begin{proof}
Lower bound $\liminf_{p \to \infty} (p+1)^{N-1} C_p \ge \( \frac{Ne}{N-1} \beta_N \)^{N-1}$ is a direct consequence of Lemma \ref{Lemma:refined_Sobolev} and the fact 
\begin{equation}
\label{C_p2}
	C_p = \frac{\| H(\nabla u_p) \|_{L^N(\Omega)}^N}{\| u_p \|_{L^{p+1}(\Omega)}^N}
\end{equation}
for least energy solutions $u_p$.

Therefore we must prove only the upper bound. 
We will do this by constructing a suitable test function for the value $C_p$.

We may assume that $0 \in \Omega$ and $\mathcal{W}_L \subset \Omega$ where $\mathcal{W}_L = \{ x \in \re^N : H^0(x) < L \}$.
For $0<l<L$, consider the Finsler Moser function
\[
	m_l(x) = \frac{1}{(N\kappa_N)^{1/N}} 
	\begin{cases}
	\( \log \frac{L}{l} \)^{\frac{N-1}{N}}, \quad &0 \le H^0(x) \le l, \\
	\frac{\log \frac{L}{H^0(x)}}{(\log \frac{L}{l})^{\frac{1}{N}}}, \quad &l \le H^0(x) \le L, \\
	0, \quad &L \le H^0(x).
	\end{cases}
\]
We check that the Moser function $m_l \in W^{1,N}_0(\Omega)$ and $\| H(\nabla m_l) \|_{L^N(\Omega)} = 1$. 
Also it is easily checked that
\[
	\(\intO m_l^{p+1} dx \)^{\frac{1}{p+1}} \ge \(\int_{\mathcal{W}_l} m_l^{p+1} dx \)^{\frac{1}{p+1}} \ge \frac{1}{(N\kappa_N)^{1/N}} \( \log \frac{L}{l} \)^{\frac{N-1}{N}} \( l^N \kappa_N \)^{\frac{1}{p+1}}.
\]
Choosing $l = L \exp \(-(\frac{N-1}{N^2}) (p+1) \)$, we have
\[
	\| m_l \|_{L^{p+1}(\Omega)} \ge \frac{1}{(N\kappa_N)^{1/N}} \(\frac{N-1}{N^2} \)^{\frac{N-1}{N}} e^{-\frac{N-1}{N}} (p+1)^{\frac{N-1}{N}} \( L^N \kappa_N \)^{\frac{1}{p+1}}.
\]
and
\[
	C_p \le \frac{\Vert H(\nabla m_l) \Vert_{L^N(\Omega)}^N}{\Vert m_l \Vert_{L^{p+1}(\Omega)}^N} \le N\kappa_N \( \frac{N^2e}{N-1} \)^{N-1} (p+1)^{-(N-1)} (L^N \kappa_N)^{-\frac{N}{p+1}},
\]
which implies $\limsup_{p \to \infty} (p+1)^{N-1} C_p \le  \( \frac{Ne}{N-1} \beta_N \)^{N-1}$.
\end{proof}

Since 
\[
	\intO H(\nabla u_p)^N dx = \intO u_p^{p+1} dx
\]
and \eqref{C_p2}, we have the following lemma.
\begin{lemma}
\label{Lemma:energy}
\[
	\lim_{p \to \infty} p^{N-1} \intO H(\nabla u_p)^N dx = \lim_{p \to \infty} p^{N-1} \intO u_p^{p+1} dx = \( \frac{Ne}{N-1} \beta_N \)^{N-1}.
\]
\end{lemma}

\section{Proof of Theorem \ref{Theorem:bound}}
\label{section:bound}

To obtain a lower bound for $\| u_p \|_{L^{\infty}(\Omega)}$,
define the first eigenvalue of the Finsler $N$-Laplacian $Q_N$:
\[
	\la_1(\Omega) = \inf \{ \intO H(\nabla u)^N dx: u \in W^{1,N}_0(\Omega), \int_{\Omega} |u|^N dx =1 \}.
\]
It is known that $0 < \la_1(\Omega) < \infty$ and
\[
	\intO u_p^{p+1} dx = \intO H(\nabla u_p)^N dx \ge \la_1(\Omega) \intO u_p^N dx.
\]
Thus 
\[
	\intO ( u_p^{p+1} - \la_1(\Omega) u_p^N) dx \ge 0,
\]
which implies 
\begin{equation}
\label{u_p_lower}
	\| u_p \|_{L^{\infty}(\Omega)}^{p+1-N} \ge \la_1(\Omega).
\end{equation}

To obtain a uniform upper bound of $\Vert u_p \Vert_{L^{\infty}(\Omega)}$, 
we use an argument with the coarea formula and the Finsler isoperimetric inequality in $\re^N$.
Set
\begin{align*}
	&\gamma_p = \max_{x \in \Omega} u_p(x), \\
	&\Omega_t = \{ x \in \Omega: u_p(x) > t \}, \\
	&\mathcal{A} = \{ x \in \Omega: u_p(x) > \frac{\gamma_p}{2} \}.
\end{align*}
By Lemma \ref{Lemma:refined_Sobolev} with $t = \frac{Np}{N-1}$ and by Lemma \ref{Lemma:energy}, we have
\[
	\( \intO u_p^{\frac{Np}{N-1}} dx \)^{\frac{N-1}{Np}} \le D_{\frac{Np}{N-1}} \(\frac{Np}{N-1}\)^{\frac{N-1}{N}} \Vert H(\nabla u_p) \Vert_{L^N(\Omega)}
	\le M
\] where $M$ is independent of $p$ if $p$ large.
From this and Chebyshev's inequality, we have
\begin{equation}
\label{Ren-Wei(3.2)}
	\(\frac{\gamma_p}{2}\)^{\frac{Np}{N-1}} |\mathcal{A}| \le M^{\frac{Np}{N-1}}.
\end{equation}

On the other hand, by approximating the constant $1$ on $\Omega_t$ by $C_0^{\infty}$-functions, we have
\[
	-\int_{\Omega_t} \diver \( H(\nabla u_p)^{N-1} (\nabla_{\xi} H)(\nabla u_p) \)  dx = \int_{\Omega_t} u_p^p dx.
\]
Thus integration by parts leads to
\begin{align}
\label{E1}
	\int_{\Omega_t} u_p^p dx &= -\int_{\pd\Omega_t} H(\nabla u_p)^{N-1} (\nabla_{\xi} H)(\nabla u_p) \cdot \nu ds \notag \\
	&= \int_{\pd\Omega_t} \frac{H(\nabla u_p)^{N-1}(\nabla_{\xi} H)(\nabla u_p) \cdot \nabla u_p}{|\nabla u_p|} ds \\
	&= \int_{\pd\Omega_t} \frac{H(\nabla u_p)^N}{|\nabla u_p|} ds, \notag
\end{align}
since the outer unit normal $\nu$ to $\pd\Omega_t$ is $\nu= -\frac{\nabla u_p}{|\nabla u_p|}$.
Here we used Proposition \ref{Prop:identities} (3) in the last equality. 
Coarea formula implies 
\[
	|\Omega_t| = \int_{\Omega_t} 1 dx = \int_t^{\infty} \int_{\{ u_p = s \}} \frac{ds}{|\nabla u_p|}.
\]
Thus
\begin{equation}
\label{E2}
	-\frac{d}{dt} |\Omega_t| = \int_{\partial\Omega_t} \frac{ds}{|\nabla u_p|}.
\end{equation}
By \eqref{E1}, \eqref{E2}, and the Schwartz inequality, we have 
\begin{align}
\label{E3}
	\( -\frac{d}{dt} |\Omega_t| \)^{N-1} \int_{\Omega_t} u_p^p dx &=  \(\int_{\pd\Omega_t} \frac{1}{|\nabla u_p|} ds \)^{N-1} \( \int_{\pd\Omega_t} \frac{H^N(\nabla u_p)}{|\nabla u_p|} ds \) \notag \\
	&\ge \( \int_{\pd\Omega_t} \frac{H(\nabla u_p)}{|\nabla u_p|} ds \)^N  \\
	&= \( \int_{\pd\Omega_t} H( \nu ) ds \)^N  \notag \\
	&= P_H(\Omega_t, \re^N)^N \ge N^N \kappa_N |\Omega_t|^{N-1}. \notag
\end{align}
In the last inequality of \eqref{E3}, we used the Finsler isoperimetric inequality in $\re^N$ \cite{AFTL}, \cite{Taylor}, \cite{Fonseca-Muller}:
\begin{equation}
\label{Finsler-isoperimetric}
	P_H(E, \re^N) \ge N \kappa_N^{\frac{1}{N}} |E|^{\frac{N-1}{N}} 
\end{equation}
for any set of finite perimeter $E \subset \re^N$ with respect to $H$.

Now, define $r(t) > 0$ such that 
\[
	|\Omega_t| = \kappa_N r^N(t).
\]
Then
\[
	\frac{d}{dt}|\Omega_t| = N \kappa_N r^{N-1}(t) r'(t).
\]
Note that $r'(t) <0$.
Putting this in \eqref{E3}, we have
\begin{align*}
	&\( -N \kappa_N r^{N-1}(t) \frac{dr}{dt}(t)  \)^{N-1} \int_{\Omega_t} u_p^p dx \ge N^N \kappa_N |\Omega_t|^{N-1}, \\
	&\( -\frac{dr}{dt} \)^{N-1} \int_{\Omega_t} u_p^p dx \ge (N\kappa_N) r^{N-1}, \\
	&-\frac{dt}{dr} \le \( \int_{\Omega_t} u_p^p dx \)^{\frac{1}{N-1}} (N\kappa_N)^{-\frac{1}{N-1}} r^{-1} \\
	&\le C r^{-1} \gamma_p^{\frac{p}{N-1}} |\Omega_t|^{\frac{1}{N-1}} = C \gamma_p^{\frac{p}{N-1}} r^{\frac{1}{N-1}},
\end{align*}
where $C$ is a constant dependent only on $N$ and varies from line to line.
Integrating the last inequality from $r=0$ to $r=r_0$, we have
\[
	t(0) - t(r_0) \le C \gamma_p^{\frac{p}{N-1}} r_0^{\frac{N}{N-1}}.
\]
Choose $r_0$ such that $t(r_0) = \frac{\gamma_p}{2}$.
Then the above inequality implies
\begin{align*}
	\gamma_p \le C \gamma_p^{\frac{p}{N-1}} r_0^{\frac{N}{N-1}}, \quad \text{i.e.,} \quad \gamma_p \le C \gamma_p^{\frac{p}{N-1}} |\mathcal{A}|^{\frac{1}{N-1}}.
\end{align*}
Combining this with \eqref{Ren-Wei(3.2)}, we have
\begin{align*}
	&\gamma_p \le C \gamma_p^{\frac{p}{N-1}} \( \frac{M^{\frac{Np}{N-1}}}{\(\frac{\gamma_p}{2}\)^{\frac{Np}{N-1}}} \)^{\frac{1}{N-1}} = C \gamma_p^{-\frac{p}{(N-1)^2}} M^{\frac{Np}{(N-1)^2}}, \\
	&\gamma_p^{1 + \frac{p}{(N-1)^2}} \le C M^{\frac{Np}{(N-1)^2}}, \\
	&\gamma_p \le C^{\frac{(N-1)^2}{(N-1)^2 + p}} M^{\frac{Np}{(N-1)^2 + p}}.
\end{align*}
From this, we conclude that there exists $C>0$ (independent of $p$) such that $\gamma_p \le C$ for $p$ large.

The latter half part of Theorem \ref{Theorem:bound} is already proven in Lemma \ref{Lemma:energy}.
Thus we have completed the proof of Theorem \ref{Theorem:bound}.
\qed

From Theorem \ref{Theorem:bound}, we have the following consequence.

\begin{corollary}
\label{Corollary:bound}
There exist $C, C' >0$ independent of $p$ large such that
\[
	C \le p^{N-1} \intO u_p^p dx \le C'
\]
holds true.
\end{corollary}

\begin{proof}
By Theorem \ref{Theorem:bound}, we have
\[
	\frac{1}{C_2} p^{N-1} \intO u_p^{p+1} dx \le \frac{\Vert u_p \Vert_{L^{\infty}(\Omega)}}{C_2} p^{N-1} \intO u_p^p dx \le p^{N-1} \intO u_p^p dx
\]
where $C_2$ is as in Theorem \ref{Theorem:bound}.
The left-hand side of the above inequality is bounded from below by a positive constant by Lemma \ref{Lemma:energy}.
On the other hand, H\"older's inequality implies
\[
	p^{N-1} \intO u_p^p dx \le \( p^{N-1} \intO u_p^{p+1} dx \)^{\frac{p}{p+1}} p^{\frac{1}{p+1}}|\Omega|^{\frac{1}{p+1}}
\]
and the right-hand side of the above inequality is bounded from above by Lemma \ref{Lemma:energy}.
This proves the conclusion.
\end{proof}

\section{Proof of Theorem \ref{Theorem:limsup}}
\label{section:limsup_bound}

In this section, we prove Theorem \ref{Theorem:limsup}.
Since $\limsup_{p \to \infty} \| u_p \|_{L^{\infty}(\Omega)} \ge 1$ immediately follows from \eqref{u_p_lower} (this is true for any solution sequence, not necessary least energy solutions),
we just need to prove $\limsup_{p \to \infty} \| u_p \|_{L^{\infty}(\Omega)} \le e^{\frac{N-1}{N}}$.
For this purpose, we follow the argument by Adimurthi and Grossi \cite{Adimurthi-Grossi}.

Let $x_p \in \Omega$ be a point so that the least energy solution to \eqref{E_N} takes its maximum: $u_p(x_p) = \| u_p \|_{L^{\infty}(\Omega)}$.
As in \cite{Adimurthi-Grossi}, We make a change of variable
\begin{equation}
	\label{z_p}
	z_p(x) = \frac{p}{u_p(x_p)} \( u_p(\eps_p x + x_p) - u_p(x_p) \), \quad x \in \Omega_p = \frac{\Omega - x_p}{\eps_p},
\end{equation}
where $\eps_p > 0$ is defined so that 
\begin{equation}
\label{ep}
	\eps_p^N p^{N-1} u_p(x_p)^{p+1-N} \equiv 1.
\end{equation}
By Theorem \ref{Theorem:bound}, we see $\eps_p \to 0$ as $p \to \infty$.
Since $u_p$ is a weak solution to \eqref{E_N}, $z_p$ is a weak solution to
\begin{align}
\label{z_p_eq}
	\begin{cases}
	&-Q_N z_p = \( 1 + \frac{z_p}{p} \)^{p} \quad \mbox{in} \; \Omega_p, \\
	&z_p |_{\pd\Omega_p} = -p, \\ 
	&\max_{x \in \ol{\Omega}_n} z_n(x) = z_n(0) = 0, \\
	&-p < z_p \le 0 \quad  \mbox{in} \; \Omega_p \\
	\end{cases}
\end{align}
by Lemma \ref{Lemma:weak}.
We want to pass to the limit as $p \to \infty$ in \eqref{z_p_eq}.
For this purpose, take any ball $B_R(0) \subset \Omega_p$ centered at the origin and radius $R$.
Consider
\begin{align*}
	\begin{cases}
	&-Q_N w_p = \( 1 + \frac{z_p}{p} \)^p \quad \mbox{in} \; B_R(0), \\
	&w_p |_{\pd B_R(0)} = 0.
	\end{cases}
\end{align*}
Comparison principle for $-Q_N$ (see for example, \cite{Xie-Gong} Theorem 3.1) and Serrin's elliptic estimate Theorem \ref{Theorem:Serrin} yield that $0 \le w_p \le C$ on $B_R(0)$ where $C$ is a constant
independent of $p$.
Set $\psi_p(x) = w_p(x) - z_p(x), x \in B_R(0)$. 
Then $\psi_p$ is a nonnegative in $B_R(0)$ and $\psi_p(0) = w_p(0) - z_p(0) = w_p(0) \le C$ uniformly in $p$.
Moreover, we have
\[
	0 = -(Q_N w_p - Q_N z_p) = -\tilde{Q}_N (w_p - z_p) = -\tilde{Q}_N \psi_p
\]
where
\begin{align*}
	&\tilde{Q}_N (w_p - z_p) \\
	&= \sum_{i,j=1}^N \frac{\pd}{\pd x_i} \left[ \int_0^1 \frac{1}{N} \frac{\pd^2 H^N}{\pd \xi_i \pd \xi_j} (t \nabla w_p + (1-t) \nabla z_p) dt \frac{\pd}{\pd x_j} (w_p(x) - z_p(x)) \right].
\end{align*}
Thanks to the assumption that ${\rm Hess} H^N(\xi)$ is positive definite, $\tilde{Q}_N$ is a quasilinear elliptic differential operator.
Thus we can apply Serrin's Harnack inequality (Theorem \ref{Theorem:Serrin} (3)) to $\psi_p$, which implies that there exists $C = C(R,r) > 0$ 
for any $0 < r < R$ such that
\[
	\sup_{B_r(0)} \psi_p(x) \le C \(1 + \inf_{x \in B_r(0)} \psi_p(x) \) \le C (1 + \psi_p(0)) =  C (1 + w_p(0)) \le C.
\]
Thus we have
\[
	0 \ge z_p(x) = w_p(x) - \psi_p(x) \ge -C
\]
for $x \in B_r(0)$. Since $0 < r < R$ is arbitrary, we have $\{ |z_p| \} \subset L^{\infty}_{loc}(B_R(0)))$ is uniformly bounded in $p$. 
Again Serrin's regularity estimate implies that $\{ z_p \}$ is bounded in $C^{1,\alpha}_{loc}(B_R(0)))$ for any $R > 0$ uniformly in $p$.

Now, we consider two cases: 

{\bf Case (i)}: $\frac{{\rm dist}(x_p, \pd\Omega_p)}{\eps_p} \to +\infty$

{\bf Case (ii)}: $\frac{{\rm dist}(x_p, \pd\Omega_p)}{\eps_p}$ is bounded and 
\[
	\Omega_p \to \re^N_{+}(s_0) = \{ x = (x', x_N) \in \re^N \, : \, x_N > s_0 \} \quad (p \to \infty)
\]
for some $s_0$.

In the case (i), note that $\Omega_p \to \re^N$ as $p \to \infty$.
Hence by the Ascoli-Arzel\'a theorem, we know that (up to a subsequence), $\{ z_p \}$ converges to some function $z \in C^1(\re^N)$ and
$z$ satisfies
\[
	-Q_N z = e^z \quad \text{in} \, \re^N.
\]

Now we claim that $\int_{\re^N} e^z dx < +\infty$.
In fact, since $z_p \to z$ in $C^1_{loc}(\re^N)$, we obtain
\[
	1_{\Omega_p}(x) \( 1 + \frac{z_p(x)}{p} \)^p \to e^{z(x)}
\]
pointwisely for $x \in \re^N$, where $1_{\Omega_p}$ is the characteristic function of $\Omega_p$. 
By using Fatou's lemma and H\"older's inequality, we deduce
\begin{align*}
	\int_{\re^N} e^z dx 
	&\le \liminf_{p \to \infty} \int_{\Omega_p} \( 1 + \frac{z_p(x)}{p} \)^p dx \\
	&\le \lim_{p \to \infty} \frac{p^{N-1}}{(u_p(x_p))^{N-1}} \intO (u_p(y))^p dy \\
	&\le \lim_{p \to \infty} \frac{p^{N-1}}{(u_p(x_p))^{N-1}} \( \intO (u_p(y))^{p+1} dy \)^{p/(p+1)} |\Omega|^{1/(p+1)} \\
	&\le C < \infty
\end{align*}
where we have used the facts that $\intO u_p^{p+1} dy = \frac{O(1)}{p^{N-1}}$ by Lemma \ref{Lemma:energy} and $u_p(x_p) \ge C_1 >0$ by Theorem \ref{Theorem:bound}.
Hence, we check that the limit function satisfies
\begin{equation}
\label{case1}
	\begin{cases}
	&-Q_N z = e^z \quad \text{in} \, \re^N, \\ 
	&z \le 0, \quad \text{in} \, \re^N, \\ 
	&\int_{\re^N} e^z dx < \infty.
	\end{cases}
\end{equation}

In the case (ii), almost the same proof works, and we see that the limit function $z$ is a solution of
\begin{equation}
\label{case2}
	\begin{cases}
	&-Q_N z = e^z \quad \text{in} \, \re^N_{+}(s_0), \\ 
	&z \le 0, \quad \text{in} \, \re^N_{+}(s_0), \\ 
	&z = -\infty, \quad \text{on} \, \pd \re^N_{+}(s_0), \\ 
	&\int_{\re^N_{+}(s_0)} e^z dx < \infty.
	\end{cases}
\end{equation}

Now we prove the following lemma.
The case $N = 2$ was treated by Ding (see \cite{Chen-Li}) when $H(\xi) = |\xi|$, and by Wang and Xia \cite{Wang-Xia} for general $H(\xi)$.

\begin{lemma}
\label{Lemma:Ding}
If $z$ is a $C^1$ weak solution of \eqref{case1}, then we have
\[
	\int_{\re^N} e^z dx \ge \(\frac{N}{N-1}\)^{N-1} N^N \kappa_N.
\]
If $z$ is a $C^1$ weak solution of \eqref{case2}, then we have
\[
	\int_{\re^N_{+}(s_0)} e^z dx \ge \(\frac{N}{N-1}\)^{N-1} N^N \kappa_N.
\]
\end{lemma}

\begin{proof}
As in the proof of Theorem \ref{Theorem:bound}, we use a level set argument.
First, we assume $z$ is a solution of \eqref{case1}.
Put
\[
	\Omega_t = \{ x \in \re^N \ : \ z(x) > t \}, \quad \mu(t) = |\Omega_t|.
\]
Integration by parts on $\Omega_t$ leads to
\begin{align*}
	\int_{\Omega_t} e^z  dx = -\int_{\Omega_t} Q_N z dx &= \int_{\pd\Omega_t} H^{N-1}(\nabla z) (\nabla_{\xi} H)(\nabla z) \cdot \frac{\nabla z}{|\nabla z|} ds_x \\
	&= \int_{\pd\Omega_t} \frac{H^N(\nabla z)}{|\nabla z|} ds_x.
\end{align*}
By the Finsler isoperimetric inequality \eqref{Finsler-isoperimetric} and H\"older's inequality, we see
\begin{align*}
	N \kappa_N^{1/N} |\Omega_t|^{\frac{N-1}{N}} \le P_H(\Omega_t, \re^N) &= \int_{\pd\Omega_t} \frac{H(\nabla z)}{|\nabla z|} ds_x \\
	&\le \(\int_{\pd\Omega_t} \frac{H^N(\nabla z)}{|\nabla z|} ds_x \)^{\frac{1}{N}} \(\int_{\pd\Omega_t} \frac{ds_x}{|\nabla z|} \)^{\frac{N-1}{N}} \\
	&= \( \int_{\Omega_t} e^z dx \)^{\frac{1}{N}} \( -\mu'(t) \)^{\frac{N-1}{N}},
\end{align*}
here we have used coarea formula
\[
	\mu(t) = \int_t^{\infty} \int_{\{ x : z(x) = s \}} \frac{ds_x}{|\nabla z|} ds.
\]
Thus we have
\[
	\mu(t) \le \left\{ \frac{1}{N \kappa_N^{1/N}} \( \int_{\Omega_t} e^z dx \)^{\frac{1}{N}} \( -\mu'(t) \)^{\frac{N-1}{N}} \right\}^{\frac{N}{N-1}}.
\]
Therefore, we obtain
\begin{align*}
	\int_{\re^N} e^z dx &= \int_{-\infty}^{\max z} e^t \mu(t) dt \\
	&\le \( \frac{1}{N \kappa_N^{1/N}} \)^{\frac{N}{N-1}} \int_{-\infty}^{\max z} e^t  \( \int_{\Omega_t} e^z dx \)^{\frac{1}{N-1}} ( -\mu'(t) ) dt \\
	&= \( \frac{1}{N \kappa_N^{1/N}} \)^{\frac{N}{N-1}} \( \frac{N-1}{N} \)  \int_{-\infty}^{\max z} \frac{d}{dt} \( \int_{\Omega_t} e^z dx \)^{\frac{N}{N-1}} dt \\
	&= \( \frac{1}{N \kappa_N^{1/N}} \)^{\frac{N}{N-1}} \( \frac{N-1}{N} \) \( \int_{\re^N} e^z dx \)^{\frac{N}{N-1}},
\end{align*}
which implies
\[
	\(\frac{N}{N-1}\)^{N-1} N^N \kappa_N \le \int_{\re^N} e^z dx.
\]

The proof when $z$ is a solution to \eqref{case2} is similar, 
since the boundary condition $z = -\infty$ on $\pd\re^N_{+}(s_0)$ assures that all level sets of $z$ are confined in $\re^N_{+}(s_0)$.
\end{proof}

By the change of variables, we have
\begin{equation}
\label{change}
	p^{N-1} \intO u_p^{p+1}(y) dy = u_p^N(x_p) \int_{\Omega_p} \(1 + \frac{z_p(x)}{p} \)^{p+1} dx.
\end{equation}
Let us take $\limsup_{p \to \infty}$ of both sides of \eqref{change}.
Then we see
\[
	\limsup_{n \to \infty} \text{LHS of } \eqref{change} = \(\frac{N e}{N-1}\)^{N-1} N^N \kappa_N
\]
by Lemma \ref{Lemma:energy}.
On the other hand, Fatou's lemma and Lemma \ref{Lemma:Ding} implies
\begin{align*}
	\limsup_{p \to \infty} \text{RHS of } \eqref{change} &\ge (\limsup_{p \to \infty} u_p(x_p))^N \times 
	\begin{cases}
	\int_{\re^N} e^z dx \quad &\text{when case (i)} \\
	\int_{\re^N_{+}(s_0)} e^z dx \quad &\text{when case (ii)}
	\end{cases} \\
	&\ge (\limsup_{p \to \infty} u_p(x_p))^N \(\frac{N}{N-1}\)^{N-1} N^N \kappa_N.
\end{align*}
Hence, we have 
\[
	e^{N-1} \ge (\limsup_{p \to \infty} \| u_p \|_{L^{\infty}(\Omega)})^N.
\]
which implies Theorem \ref{Theorem:limsup}
\qed

\section{Proof of Theorem \ref{Theorem:blowup}}
\label{section:blowup}

In this section, we prove Theorem \ref{Theorem:blowup}.
Given any sequence $p_n$ of $p$ with $p_n \to \infty$, 
let us recall \eqref{v_p} and \eqref{L_0L_1} for $p = p_n$, $u_n = u_{p_n}$.
\begin{align*}
	&v_n = \frac{u_n}{\la_n} = \frac{u_n}{(\intO u_n^{p_n} dx)^{\frac{1}{N-1}}}, \quad \la_n = \( \intO u_n^{p_n} dx \)^{\frac{1}{N-1}}, \\ 
	&f_n(x) = \frac{u_n^{p_n}}{\intO u_n^{p_n} dx}, \\
	&L_0 = \limsup_{n \to \infty} \frac{p_n \(\intO u_n^{p_n} dx \)^{\frac{1}{N-1}}}{\( \frac{N}{N-1} e^{\frac{N-1}{N}} \)}, \quad L_1 = d_N^{-\(\frac{1}{N-1}\)} L_0.
\end{align*}
Then $v_n$ is a weak solution of \eqref{v_p_eq} for $p = p_n$.
By H\"older's inequality and Theorem \ref{Theorem:bound}, we see 
\[
	p_n^{N-1} \intO u_n^{p_n} dx \le p_n^{N-1} \(\intO u_n^{p_n+1} dx \)^{\frac{p_n}{p_n+1}} |\Omega|^{\frac{1}{p_n+1}} \to \( \frac{Ne}{N-1} \beta_N \)^{N-1}
\] 
as $n \to \infty$.
This shows that
\[
	L_0 \le e^{\frac{1}{N}} \beta_N, \quad L_1 \le e^{\frac{1}{N}} \beta_N d_N^{-(\frac{1}{N-1})}.
\]

First, we prove $S \ne \void$ for any sequence $v_n = v_{p_n}$ of $v_p$ with $p_n \to \infty$. 
Indeed, by Theorem \ref{Theorem:bound}, we have $\| u_n \|_{L^{\infty}(\Omega)} \ge C_1 > 0$ for any $n \in \N$.
Let $x_n \in \Omega$ be a point such that $u_n(x_n) = \| u_n \|_{L^{\infty}(\Omega)}$, then
\[
	v_n(x_n) = \frac{u_n(x_n)}{(\intO u_n^{p_n} dx)^{\frac{1}{N-1}}} \ge \frac{C_1}{(\intO u_n^{p_n} dx)^{\frac{1}{N-1}}} = \frac{C_1}{O(\frac{1}{p_n})} \to +\infty
\]
by Lemma \ref{Lemma:energy}.
This implies that any accumulation point of $\{ x_n \}$ is contained in $S$ and hence $S \ne \void$.

Next, as in \cite{Brezis-Merle}, \cite{Ren-Wei(TAMS)}, \cite{Ren-Wei(PAMS)}, we define {\it $(L, \delta)$-regular set} and {\it $(L, \delta)$-irregular set} of a sequence $\{ u_n \}$.
Since
\[
	f_n = \frac{u_n^{p_n}}{\intO u_n^{p_n} dx} \in L^1(\Omega), \quad f_n \ge 0, \quad \intO f_n dx = 1,
\]
there exists a subsequence (still denoted by $n$) such that
\[
	f_n \stackrel{*}{\weakto} \mu, \quad \mu(\Omega) \le 1
\]
in the sense of Radon measures of $\Omega$, where $\mu$ is a nonnegative Radon measure.

Given $L > 0$ and $\delta >0$, we call a point $x_0 \in \Omega$ a {\it $(L, \delta)$-regular point of} $\{ u_n \}$ if
there exists $\varphi \in C_0(\Omega), 0\le \varphi \le 1$ with $\varphi \equiv 1$ near $x_0$ such that
\[
	\intO \varphi d\mu < \( \frac{\beta_N}{L+3\delta} \)^{N-1}
\]
where $\beta_N = N(N\kappa_N)^{\frac{1}{N-1}}$ is as in Theorem \ref{Theorem:TM}.
We put 
\begin{align*}
	&R_L(\delta) = \{ x_0 \in \Omega: x_0 \; \mbox{is a} \; (L, \delta)\text{-regular point} \}, \\
	&\Sigma_L(\delta) = \Omega \setminus R_L(\delta).
\end{align*}
We call a point in $\Sigma_L(\delta)$ an {\it $(L, \delta)$-irregular point} of the sequence $\{ u_n \}$.
Note that $(L, \delta)$-regular, or $(L, \delta)$-irregular points are automatically interior points of $\Omega$.
Also note that if $x_0 \in \Sigma_L(\delta)$, then we have
\begin{equation}
\label{atom}
	\mu(\{ x_0 \}) \ge \( \frac{\beta_N}{L+3\delta} \)^{N-1}.
\end{equation}
Since 
\[
	1 \ge \mu(\Omega) \ge \( \frac{\beta_N}{L+3\delta} \)^{N-1} \sharp \Sigma_L(\delta)
\]
by \eqref{atom}, we see that $\Sigma_L(\delta)$ is a finite set for any $L > 0$ and $\delta > 0$.

Next Lemma is the key to analyze the interior blow-up set $S \cap \Omega$. 

\begin{lemma}{\rm (smallness of $\mu$ implies boundedness)}
\label{Lemma:small}
Let $x_0$ be a $(L_1, \delta)$-regular point of a sequence $\{ u_n \}$ where $L_1$ is defined in \eqref{L_0L_1}.
Then $\{ v_n \}$ is bounded in $L^{\infty}(B_{R_0}(x_0))$ for some $R_0>0$.
\end{lemma}

\begin{proof}
Key point in the proof is to get the following pointwise estimate
\begin{equation}
\label{pointwise}
	f_n(x) < \exp \( (L_1 + \delta/2) d_N^{\frac{1}{N-1}} v_n(x) \), \quad x \in \Omega.
\end{equation}

In checking (\ref{pointwise}), we use the elementary inequality
\begin{equation}
\label{elementary}
	\frac{\log x}{x} \le \frac{\log y}{y} \quad \text{for any} \; 0 < x \le y \le e.
\end{equation}
Let 
\[
	\alpha_n = \frac{\| u_n \|_{L^{\infty}(\Omega)}}{\( \intO u_n^{p_n} dx \)^{\frac{1}{p_n}}} = \frac{\| u_n \|_{L^{\infty}(\Omega)}}{\la_n^{\frac{N-1}{p_n}}},
\]
and recall that $\la_n = O\( \frac{1}{p_n} \)$ by Corollary \ref{Corollary:bound}, so $\la_n^{\frac{N-1}{p_n}} = O\( \frac{1}{p_n}\)^{\frac{N-1}{p_n}} \to 1$ as $n \to \infty$. 
Thus we have
\[
	\limsup_{n \to \infty} \alpha_n = \limsup_{n \to \infty} \| u_n \|_{L^{\infty}(\Omega)} \le e^{\frac{N-1}{N}}
\]
by Theorem \ref{Theorem:limsup}.
From this, we see that for any small $\eps' > 0$,
\[
	\frac{u_n(x)}{\la_n^{\frac{N-1}{p_n}}} \le \alpha_n \le e^{\frac{N-1}{N}} + \eps' < e
\]
holds for any $x \in \Omega$ and for large $n$.
Therefore by \eqref{elementary}, we have for fixed small $\eps>0$ 
\[
	\frac{\log \( \frac{u_n(x)}{\la_n^{\frac{N-1}{p_n}}} \)}{\frac{u_n(x)}{\la_n^{\frac{N-1}{p_n}}}} 
	\le \frac{\log \alpha_n}{\alpha_n}
	\le \( \frac{N-1}{N} \) \frac{1}{e^{\frac{N-1}{N}}} + \eps
\]
for large $n$.
Hence
\begin{align*}
	\log f_n(x) = p_n \log \frac{u_n(x)}{\la_n^{\frac{N-1}{p_n}}} 
	&\le p_n \( \frac{u_n(x)}{\la_n^{\frac{N-1}{p_n}}} \) \( \frac{N-1}{N} \frac{1}{e^{\frac{N-1}{N}}} + \eps \) \\
	&= p_n \la_n \( \frac{N-1}{N} e^{-\frac{N-1}{N}} + \eps \) \frac{v_n(x)}{\la_n^{\frac{N-1}{p_n}}} \\
	&\le \( \frac{N}{N-1} e^{\frac{N-1}{N}} L_1 d_N^{\frac{1}{N-1}} + \eps \) \( \frac{N-1}{N} e^{-\frac{N-1}{N}} + 2\eps \) v_n(x),
\end{align*}
here we have used $\lim_{n \to \infty} \la_n^{\frac{N-1}{p_n}} = 1$ and
\[
	p_n \la_n \le \frac{N}{N-1} e^{\frac{N-1}{N}} L_1 d_N^{\frac{1}{N-1}} + \eps
\]
for large $n$ by the definition of $L_1$.
Therefore 
\[
	\log f_n(x) \le \( (L_1 + \delta/2) d_N^{\frac{1}{N-1}} \) v_n(x)
\]
holds if we choose $\eps > 0$ small enough.
This proves the pointwise estimate (\ref{pointwise}).

Next, by the use of Brezis-Merle theory for the Finsler $N$-Laplacian, we obtain the integral estimate
\begin{equation}
\label{integral_estimate}
	\int_{B_{R_1/2}(x_0)} \exp \( (L_1 + \delta) d_N^{\frac{1}{N-1}} v_n(x) \) dx \le C 
\end{equation}
for some $R_1 > 0$ small and $C > 0$ independent of $n$, here $x_0$ is a $(L_1, \delta)$-regular point.

Indeed, by the definition of $(L_1, \delta)$-regular point,
we can find $R_1 > 0$ such that
\[
	\int_{B_{R_1}(x_0)} f_n dx \le  \(\frac{\beta_N}{L_1 + 2\delta} \)^{N-1}.
\]
Also by Theorem \ref{Theorem:BM} (i) and the fact that $\| f_n \|_{L^1(\Omega)} = 1$, we have 
\[
	\intO \exp \( (\beta_N - \eps) v_n(x) \) dx \le \frac{\beta_N}{\eps} |\Omega|
\]
for any $\eps \in (0, \beta_N)$.
From this, we obtain
\begin{equation}
\label{L^N_estimate}
	\| v_n \|_{L^N(\Omega)} \le C
\end{equation}
where $C > 0$ is independent of $n$.
Next, let $\phi_n$ be a weak solution of
\[
	-Q_N \phi_n = 0 \quad \text{in} \, B_{R_1}(x_0) \quad \phi_n = v_n \quad \text{on} \, \pd B_{R_1}(x_0).
\]
Then by Theorem \ref{Theorem:BM} (2) and the fact that $\| f_n \|_{L^1(B_{R_1}(x_0))}^{\frac{1}{N-1}} < \frac{\beta_N}{L_1 + 2\delta}$, 
we have 
\begin{equation}
\label{BM2}
	\intO \exp \( (L_1 + \delta) d_N^{\frac{1}{N-1}} |v_n(x) - \phi_n(x)| \) dx \le C
\end{equation}
if we choose $\eps \in (0, \beta_N)$ sufficiently small.
By the comparison principle for the Finsler $N$-Laplacian (see \cite{Xie-Gong} Theorem 3.2) and Serrin's estimates Theorem \ref{Theorem:Serrin} (i),
we have
\[
	\| \phi_n \|_{L^{\infty}(B_{R_1/2}(x_0))} \le \| v_n \|_{L^{\infty}(B_{R_1/2}(x_0))} \le C \| v_n \|_{L^N(B_{R_1}(x_0))} \le C
\]
where we have used \eqref{L^N_estimate}.
Combining this with \eqref{BM2}, we obtain the desired integral estimate \eqref{integral_estimate}.

Comparing \eqref{integral_estimate} and \eqref{pointwise}, we see that
$f_n$ is bounded uniformly in $n$ in $L^q(B_{R_1/2}(x_0))$ where $q = \frac{L_1 + \delta}{L_1 + \delta/2} > 1$. 
Therefore, Serrin's regularity estimate Theorem \ref{Theorem:Serrin} (i) again implies that 
\[
	\| v_n \|_{L^{\infty}(B_{R_1/4}(x_0))} \le C
\]
independent of $n$.
Taking $R_0 = R_1/4$ ends the proof of Lemma \ref{Lemma:small}.
\end{proof}

We know that $\Sigma_{L_1}(\delta)$ is a set of finite points, all of those are interior of $\Omega$.
From Lemma \ref{Lemma:small}, we obtain $S \cap \Omega = \Sigma_{L_1}(\delta)$ for any $\delta>0$ and
\[
	1 \ge \mu(\Omega) \ge \( \frac{\beta_N}{L_1 + 3\delta} \)^{N-1} \sharp (\Sigma_{L_1}(\delta)) = \( \frac{\beta_N}{L_1 + 3\delta} \)^{N-1} \sharp (S \cap \Omega).
\]
Hence
\[
	\sharp (S \cap \Omega) \le \( \frac{L_1+3\delta}{\beta_N} \)^{N-1} \le \( \frac{e^{\frac{1}{N}} \beta_N d_N^{-(\frac{1}{N-1})}+3\delta}{\beta_N} \)^{N-1}.
\]
Taking a limit $\delta \to 0$, we have
\[
	\sharp (S \cap \Omega) \le e^{\frac{N-1}{N}} d_N^{-1}
\]
This proves the first part of Theorem \ref{Theorem:blowup}.

If $x_0 \in S \cap \Omega = \Sigma_{L_1}(\delta)$, then for any  $R >0$, we have
\begin{equation}
\label{blowup}
	\lim_{n \to \infty} \| v_n \|_{L^{\infty}(B_R(x_0))} = +\infty.
\end{equation}
Indeed, if for some $R > 0$, assume there exists $C > 0$ independent of $n$ such that $\| v_n \|_{L^{\infty}(B_R(x_0))} \le C$ for all large $n$.
Then
\[
	f_n = \frac{v_n^{p_n}}{\la_n^{N-1-p_n}} \le C^{p_n} O\(\frac{1}{p_n}\)^{p_n - (N-1)} \to 0 \quad (n \to \infty)
\]
uniformly on $B_R(x_0)$.
This implies $x_0$ is a $(L_1, \delta)$-regular point, which is absurd.
The same kind of argument leads to that the limit measure $\mu$ is atomic and of the form
\[
	\mu = \sum_{i=1}^k \gamma_i \delta_{x_i}
\]
where $S \cap \Omega = \{ x_1, \cdots, x_k \}$.
Since $\mu(\Omega) \le 1$, we have $\sum_{i=1}^k \gamma_i \le 1$ and 
\[
	\gamma_i \ge (\frac{\beta_N}{L_1})^{N-1}
\]
for all $i=1,\cdots, k$ by letting $\delta \to 0$ in \eqref{atom} with $L = L_1$.
This proves Theorem \ref{Theorem:blowup} (i).

On any compact sets in $\Omega \setminus (S \cap \Omega)$, $\{ v_n \}$ is uniformly bounded.
Then by Serrin's and Tolksdorf's regularity estimate, $\{ v_n \}$ is also bounded in $C^{1, \al}_{loc}(\Omega \setminus (S \cap \Omega))$ for some $\al \in (0,1)$.
By Ascoli-Arzel\'a theorem, we have a subsequence and a function $G$ such that $v_n \to G$ in $C^1_{loc}(\Omega \setminus (S \cap \Omega))$.
That this $G$ satisfies Theorem \ref{Theorem:blowup} (ii) is clear.

Finally, since $\la_n = O(\frac{1}{p_n})$ as $n \to \infty$ and $v_n(x) = \frac{u_n(x)}{\la_n}$ is uniformly bounded in $L^{\infty}_{loc}(\Omega \setminus (S \cap \Omega))$,
we easily see that Theorem \ref{Theorem:blowup} (iii) holds.

Thus all the proof of Theorem \ref{Theorem:blowup} has been completed.

\section{Proof of Theorem \ref{Theorem:one_point_blowup}}
\label{section:one_point_blowup}

In this section, we prove Theorem \ref{Theorem:one_point_blowup}.
\begin{proof}
Assume the contrary that $x_0 \in \pd\Omega$,
where $x_0$ is the unique blow-up point of a sequence $v_n = v_{p_n}$ with $p_n \to +\infty$ as $n \to \infty$.
For $R > 0$ small, we may use the Pohozaev identity Theorem \ref{Theorem:Pohozaev} on $\Omega \cap B_R(x_0)$, with the aid of Theorem \ref{Theorem:Serrin}:
\begin{align}
\label{local_Pohozaev}
	&\frac{N}{p_n+1} \int_{\Omega \cap B_R(x_0)} u_n^{p_n+1} dx 
	= \int_{\pd (\Omega \cap B_R(x_0))} \frac{u_n^{p_n+1}}{p_n+1}(x-y) \cdot \nu(x) ds_x \\
	&- \frac{1}{N} \int_{\pd(\Omega \cap B_R(x_0))} H^N(\nabla u_n)(x-y) \cdot \nu(x) ds_x \notag \\
	&+ \int_{\pd(\Omega \cap B_R(x_0))} \( H^{N-1}(\nabla u_n) (\nabla_{\xi} H)(\nabla u_n) \cdot \nu(x) \) (x-y) \cdot \nu(x) ds_x. \notag
\end{align}
In order to remove the integral terms involving $\pd\Omega$, we use a trick in \cite{Santra-Wei}.
Define
\[
	\rho_n = \frac{\int_{\pd\Omega \cap B_R(x_0)} H^N(\nabla u_n) (x - x_0) \cdot \nu(x) ds_x}
	{\int_{\pd\Omega \cap B_R(x_0)} H^N(\nabla u_n) \nu(x_0) \cdot \nu(x) ds_x}
\]
and put $y_n = x_0 + \rho_n \nu(x_0)$. 
We assume $R > 0$ so small such that $1/2 \le \nu(x_0) \cdot \nu(x) \le 1$ for $x \in \pd\Omega \cap B_R(x_0)$.
Then we have that $\rho_n \le 2R$.
By the definition of $y_n$ and $\rho_n$, we see that
\[
	\int_{\pd\Omega \cap B_R(x_0)} H^N(\nabla u_n)(x-y_n) \cdot \nu(x) ds_x \equiv 0
\]
for all $n \in \N$.
Also since $u_n = 0$ on $\pd\Omega$ and $u_n > 0$ in $\Omega$, we see $\nu(x) = -\frac{\nabla u_n(x)}{|\nabla u_n(x)|}$.
By using these, we see \eqref{local_Pohozaev} with $y = y_n$ becomes
\begin{align}
\label{local_Pohozaev2}
	&\frac{N}{p_n+1} \int_{\Omega \cap B_R(x_0)} u_n^{p_n+1} dx 
	= \frac{1}{p_n+1} \int_{\Omega \cap \pd B_R(x_0)} u_n^{p_n+1} (x-y_n) \cdot \nu(x) ds_x \\
	&- \frac{1}{N} \int_{\Omega \cap \pd B_R(x_0)} H^N(\nabla u_n)(x-y_n) \cdot \nu(x) ds_x \notag \\
	&+ \int_{\Omega \cap \pd B_R(x_0))} \( H^{N-1}(\nabla u_n) (\nabla_{\xi} H)(\nabla u_n) \cdot \nu(x) \) (x-y_n) \cdot \nabla u_n(x) ds_x. \notag
\end{align}
Multiplying $(\frac{1}{\la_n})^N$ to both sides of \eqref{local_Pohozaev2} and recalling $v_n = \frac{u_n}{\la_n}$, we have
\begin{align}
\label{local_Pohozaev3}
	&\frac{N}{p_n+1} \( \frac{1}{\la_n} \)^N \int_{\Omega \cap B_R(x_0)} u_n^{p_n+1} dx \\
	&= \frac{1}{p_n+1} \( \frac{1}{\la_n} \)^N \int_{\Omega \cap \pd B_R(x_0)} u_n^{p_n+1} (x-y_n) \cdot \nu(x) ds_x \notag \\
	&- \frac{1}{N} \int_{\Omega \cap \pd B_R(x_0)} H^N(\nabla v_n)(x-y_n) \cdot \nu(x) ds_x \notag \\
	&+ \int_{\Omega \cap \pd B_R(x_0))} \( H^{N-1}(\nabla v_n) (\nabla_{\xi} H)(\nabla v_n) \cdot \nu(x) \) (x-y_n) \cdot \nabla v_n(x) ds_x 
	\notag \\
	& = I + II + III. \notag
\end{align}
We estimate the terms $I, II, III$ on the right-hand side of \eqref{local_Pohozaev3} as follows:
\begin{align*}
	&|I| =  \frac{1}{p_n+1} \( \frac{1}{\la_n} \)^N \big| \int_{\Omega \cap \pd B_R(x_0)} u_n^{p_n+1} (x-y_n) \cdot \nu(x) ds_x \big| \notag \\
	&\le \frac{O(p_n^N)}{p_n^{N-1}(p_n+1)} \| p_n^{N-1} u_n^{p_n+1} \|_{L^{\infty}(\Omega \cap \pd B_R(x_0))} 
\int_{\Omega \cap \pd B_R(x_0)} |(x-y_n) \cdot \nu(x) | ds_x \notag \\
	&= \frac{O(p_n^N)}{p_n^{N-1}(p_n+1)} \| p_n^{N-1} u_n^{p_n+1} \|_{L^{\infty}(\Omega \cap \pd B_R(x_0))} O(R^{N-1}).
\end{align*}
We note that since $S \cap \Omega = \void$ by assumption,
\[
	f_n = \frac{u_n^{p_n}}{\la_n^{N-1}} \to 0
\]
uniformly on compact sets in $\Omega$ and
\[
	p_n^{N-1} u_n^{p_n+1}(x) \le \| u_n \|_{L^{\infty}(\Omega)} p_n^{N-1} u_n^{p_n}(x) \le C \frac{u_n^{p_n}(x)}{\la_n^{N-1}} \le C f_n(x)
\]
by Theorem \ref{Theorem:bound} and the fact that $\la_n = O(\frac{1}{p_n})$ as $n \to \infty$.
Thus we have
\[
	\| p_n^{N-1} u_n^{p_n+1} \|_{L^{\infty}(\Omega \cap \pd B_R(x_0))} \to 0 \quad \text{as} \quad n \to \infty
\]
and thus
\[
	\lim_{R \to 0} \lim_{n \to \infty} |I| = 0.
\]
	
Also, by Theorem \ref{Theorem:blowup} (ii), we have
$v_n \to G$ in $C^{1,\al}_{loc}(\Omega \setminus (S \cap \Omega))$.
Thus we have $H^N(\nabla v_n) = O(1)$ on $\Omega \cap \pd B_R(x_0)$, which implies 
\begin{align*}
	|II| &= \frac{1}{N} \Big| \int_{\Omega \cap \pd B_R(x_0)} H^N(\nabla v_n)(x-y_n) \cdot \nu(x) ds_x \Big| \\
	&\le O(1) \int_{\Omega \cap \pd B_R(x_0)} |(x-y_n) \cdot \nu(x)| ds_x \le O(1) O(R^{N-1}), \\
	|III| &= \Big| \int_{\Omega \cap \pd B_R(x_0))} \( H^{N-1}(\nabla v_n) (\nabla_{\xi} H)(\nabla v_n) \cdot \nu(x) \) (x-y_n) \cdot \nabla v_n(x) ds_x \Big| \\
	&\le O(1) \int_{\Omega \cap \pd B_R(x_0)} |(x-y_n) \cdot \nu(x)| ds_x \le O(1) O(R^{N-1}).
\end{align*}
Therefore we have
\[
	\lim_{R \to 0} \lim_{n \to \infty} |II| = \lim_{R \to 0} \lim_{n \to \infty} |III| = 0.
\]
From these, we obtain
\begin{equation}
\label{RHS_limit}
	\lim_{R \to 0} \lim_{n \to \infty} (\text{RHS of \eqref{local_Pohozaev3}}) = 0.
\end{equation}

On the other hand, 
recall
\begin{align*}
	z_n(x) &= \frac{p_n}{u_n(x_n)} \( u_n(\eps_n x + x_n) - u_n(x_n) \), \\ 
	&x \in \Omega_{R,n} = \frac{( \Omega \cap B_R(x_0) ) - x_n}{\eps_n},
\end{align*}
where $\eps_n^N p_n^{N-1} u_n(x_n)^{p_n+1-N} \equiv 1$.
Then we see from Fatou's lemma, Theorem \ref{Theorem:bound}, and Lemma \ref{Lemma:Ding}, that
\begin{align*}
	&\lim_{R \to 0} \lim_{n \to \infty} \int_{\Omega \cap B_R(x_0))} p_n^{N-1} u_n^{p_n+1}(y) dy \\
	&= \lim_{R \to 0} \lim_{n \to \infty} u_n(x_n)^N \int_{\Omega_{R,n}} \( 1 + \frac{z_n(x)}{p_n} \)^{p_n+1} dx \\
	&\ge C_1^N \int_{U} e^z dx \ge C_1^N \(\frac{N}{N-1}\)^{N-1} N^N \kappa_N
\end{align*}
where $u = \re^N$ or $\re^N_{+}(s_0)$ for some $s_0 >0$ according to the cases
$\frac{{\rm dist}(x_n, \pd\Omega_{R,n})}{\eps_n} \to +\infty$ or 
$\frac{{\rm dist}(x_n, \pd\Omega_{R,n})}{\eps_n} \to s_0$.
Note that our assumption $\sharp S = 1$ assures that we can choose $x_n$ as a maximum points of $u_n$.
From this and the fact that $\la_n = O(\frac{1}{p_n})$ as $n \to \infty$, we have
\begin{equation}
\label{LHS_limit}
	\lim_{R \to 0} \lim_{n \to \infty} (\text{LHS of \eqref{local_Pohozaev3}}) \ge C > 0
\end{equation}
for some positive constant $C > 0$ independent of $n$.

Clearly \eqref{LHS_limit} contradicts to \eqref{RHS_limit}, and we conclude that $x_0 \not\in \pd\Omega$.
\end{proof}

Finally, as a corollary, we prove the following. 

\begin{corollary}
\label{Corollary:Wulff}
Let $R > 0$ and let $\{ u_p \}$ be a sequence of least energy solutions to
\begin{equation}
\label{E_W}
	\begin{cases}
		&-Q_N u_p = u_p^p \quad \mbox{in} \; \mathcal{W}_R, \\
		&u_p > 0 \quad \mbox{in} \; \mathcal{W}_R, \\
		&u_p = 0 \quad \mbox{on} \; \pd\mathcal{W}_R
	\end{cases}
\end{equation}
where $\mathcal{W}_R = \{ x \in \re^N \, : \, H^0(x) < R \}$.
Then the blow-up set $S$ of $v_p$ satisfies $S \cap \mathcal{W}_R = \{ 0 \}$, and
\begin{align*}
	u_p \to G(\cdot, 0) \quad \text{in } C^1_{loc}(\mathcal{W}_R \setminus \{ 0 \})
\end{align*} 
where $G$ is the unique Green function on $\mathcal{W}_R$ obtained in Theorem \ref{Theorem:Green}, and
\begin{align*}
	f_p = \frac{u_p^p}{\int_{\mathcal{W}_R} u_p^p dx} \stackrel{*}{\weakto} \delta_{0}
\end{align*}
in the sense of Radon measures on $\mathcal{W}_R$, along the full sequence.
\end{corollary}

\begin{proof}
The usual method of moving plane to prove the symmetry of solutions is not applicable in the anisotropic situation.
However, we can use Theorem 4.1 in \cite{Belloni-Ferone-Kawohl} under the convexity and $C^1$-assumption of the map $\xi \mapsto H^N(\xi)$. 
(Note that the key point of the proof of Theorem 4.1 in \cite{Belloni-Ferone-Kawohl} is the Pohozaev identity Theorem 4.2 in \cite{Belloni-Ferone-Kawohl} for $C^1(\ol{\Omega})$-weak solutions, 
which is valid by the above assumptions).
Thus we assure that any positive solution $u_p$ to \eqref{E_W} is {\it Finsler-radial}, 
that is, all level sets of $u_p$ are homothetic to $\mathcal{W}_R$ for any $p > 1$. 
Let $S$ be the blow-up set of $v_p$. Then we see that $S \cap \mathcal{W}_R = \{ 0 \}$.
Indeed, if there were a point $x_0 \in S \cap \mathcal{W}_R$, then all points on the level set of $u_p$ passing through $x_0$ must be blow-up points of $v_p$,
which contradicts to the fact that $\sharp (S \cap \mathcal{W}_R)$ is finite.
Thus by Theorem \ref{Theorem:blowup}, we see
\begin{align*}
	v_p \to G(\cdot, 0) \quad \text{in } C^1_{loc}(\mathcal{W}_R \setminus \{ 0 \})
\end{align*} 
for some function $G$ along a subsequence.
The limit function must be the unique Green function constructed in Theorem \ref{Theorem:Green},
and by the uniqueness, the convergence is true for the full sequence.
\end{proof}

%
%

\vspace{1em}\noindent
{\bf Acknowledgments.}

This work was partly supported by Osaka City University Advanced Mathematical Institute MEXT Joint Usage / Research Center on Mathematics and Theoretical Physics JPMXP0619217849.
The second author (F.T.) was supported by JSPS KAKENHI Grant-in-Aid for Scientific Research (B), JP19H01800, and JSPS Grant-in-Aid for Scientific Research (S), JP19H05597.


\end{document}